\documentclass{article}
\usepackage{graphicx} 
\usepackage{amsmath, amsthm, amsfonts, amssymb, amscd, bm, enumerate}
\usepackage{verbatim}
\usepackage{url}

\newtheorem{theorem}{Theorem}

\newtheorem{corollary}{Corollary}
\newtheorem{lemma}{Lemma}
\newtheorem{proposition}{Proposition}
\newtheorem{definition}{Definition}
\newtheorem{assumption}{Assumption}
\usepackage{csquotes}

\theoremstyle{remark}
\newtheorem{remark}{Remark}
\newtheorem{example}{Example}
\def\F{\mathcal{F}}

\def\CE{\mathcal{E}}

\def\CL{\mathcal{L}}
\def\CM{\mathcal{M}}
\def\CQ{\mathcal{Q}}
\def\CR{\mathcal{R}}

\def\bR{\mathbb{R}}
\def\bN{\mathbb{N}}
\def\bx{\mathbf{x}}

\def\ES{\mathrm{ES}}

\begin{document}
\title{Data and uncertainty in extreme risks -- a nonlinear expectations approach}

\author{Samuel N. Cohen\\
Mathematical Institute, University of Oxford\\
samuel.cohen@maths.ox.ac.uk}

\maketitle
\begin{abstract}
Estimation of tail quantities, such as expected shortfall or Value at Risk, is a difficult problem. We show how the theory of nonlinear expectations, in particular the Data-robust expectation introduced in \cite{Cohen2016}, can assist in the quantification of statistical uncertainty for these problems. However, when we are in a heavy-tailed context (in particular when our data are described by a Pareto distribution, as is common in much of extreme value theory), the theory of \cite{Cohen2016} is insufficient, and requires an additional regularization step which we introduce. By asking whether this regularization is possible, we obtain a qualitative requirement for reliable estimation of tail quantities and risk measures, in a Pareto setting.
\end{abstract}



\section{Introduction}
Statistics is, in part, the art of combining data and probabilistic models to predict the unknown. This necessarily results in uncertainty, not only as described by the estimated probabilistic model\footnote{Since the work of Knight \cite{Knight1921}, uncertainty that is described by a probabilistic model is sometimes called `risk', while lack of knowledge of the model is called `uncertainty'.}, but also in the correctness of the model itself. In many situations, it is important to describe this uncertainty, and to incorporate it in our decision making.

One area where these issues are particularly prominent is when considering extreme outcomes. By their very nature, this involves looking at unlikely events, often beyond the range of available data. Justifying this extrapolation from observations is a key purpose of `extreme value theory'. In particular, the Fisher--Tippett--Gnedenko and Pickands--Balkema--de Haan theorems tell us that, assuming they exist, the possible distributions of renormalized extremes are  in a small class, allowing asymptotic approximation methods to be used. Nevertheless, one must still estimate parameters of these distributions, and intuition suggests that there remains a significant level of uncertainty when working with these models.

The question of how to incorporate uncertainty in decision making has a long history. Bayesian approaches to statistics do this by placing a distribution over the unknown quantities, conditioning on observations, and then integrating over the possible unknowns. This effectively forces all uncertainty to be treated `linearly', and one cannot describe an agent's aversion to uncertainty separately to their aversion to risk (i.e. to the randomness which is included within the model). Conversely, classical methods in statistics allow for uncertainty in parameters (for example, through the use of confidence sets), but these are frequently used with relatively limited axiomatic decision-theoretic support.

One approach, building from an axiomatic approach to valuations of outcomes, is `nonlinear expectations'. These are mathematical generalizations of the classical expected value, allowing for risk and uncertainty to be represented in a rigorous and precise manner. Nonlinear expectations can be seen as generalizations of worst-case valuation, considered in a statistical context by Wald \cite{Wald1945}, but with an increased focus on making valuations of \emph{multiple} random outcomes, rather than on minimizing a particular fixed risk. In finance, nonlinear expectations have played a particular role in the theory of convex risk measures, see for example F\"ollmer and Schied \cite{Follmer2002}. In a recent paper \cite{Cohen2016}, a method of connecting these nonlinear expectations with statistical estimation was introduced. 

In this paper, we shall look at some of the implications that this way of thinking has on extreme value theory. In particular, we will consider how quantifying our uncertainty leads to scepticism in extreme extrapolation, and how different approaches to measuring the riskiness of a decision require different quantities of data.

This paper proceeds as follows. Following an initial summary of the theory of nonlinear expectations and their connection with estimation, we will consider how to combine these expectations with measures of risk, corresponding in some cases to extrapolation into the tail. We will then look at methods of regularizing these nonlinear expectations, which are often needed when considering unbounded outcomes. Finally, we will apply this theory to a classic estimation problem with heavy tails, and draw conclusions for the estimation of a variety of measures of risk.

\begin{remark}
 In their well known textbook on extreme value theory in finance and insurance, Embrechts, Kl\"uppelberg and Mikosch \cite{Embrechts1997} consider the challenges of extrapolation. In particular, they discuss their preference for looking at only moderately high quantiles of distributions, rather than estimating deeper into the tail. They state,
\begin{displayquote}
 The reason why we are \emph{very reluctant} to produce plots for high quantiles like 0.9999 or more, is that we feel that such estimates are to be treated \emph{with extreme care}. [...] The statistical reliability of these estimates becomes, as we have seen, very difficult to judge in general. Though we can work out approximate confidence intervals for these estimators, such constructions strongly rely on mathematical assumptions which are unverifiable in practice. 

\hfill --- Embrechts, Kl\"uppelberg and Mikosch, \cite[p.364]{Embrechts1997}

\hfill (emphasis in original)
\end{displayquote}

In this paper, we give a partial attempt at a formal analysis supporting this reticence. We shall see that the theory of nonlinear expectations gives certain bounds (which depend on the amount of data at hand) beyond which the uncertainties from estimation dominate the computation of high quantiles (and similar quantities), hence statistical inference is problematic.
\end{remark}

\section{DR-expectations}

Consider the following problem. Suppose we have a family $\{X\}\cup\{X_i\}_{i\in \bN}$ of real-valued random variables on the canonical space $(\Omega, \F) = (\bR^\bN, \mathcal{B}(\bR^\bN))$, where $ \mathcal{B}(\bR^\bN)$ denotes the Borel cylinder $\sigma$-algebra.  We observe $\bx_N=\{X_n\}_{n=1}^N$, and seek to draw conclusions about the likely values of $\phi(X)$, where $\phi$ is some (known, Borel measurable) real function. Our aim, therefore, is to estimate the distribution of $X$, accounting for the uncertainty in our estimation.

 Beginning classically, we first propose a family $\mathcal{Q}$ of measures, assumed equivalent on $\sigma(\bx_N)$ for each $N<\infty$, corresponding to possible joint distributions of $\{X\}\cup\{X_n\}_{n\in \bN}$.
  For each of these distributions,  we obtain the log-likelihood
\[\ell(\bx_N;Q) = \log\Big(\frac{dQ|_{\sigma(\bx_N)}}{dQ_{\mathrm{ref}}|_{\sigma(\bx_N)}}\Big)(\bx_N)\] 
(relative to some reference measure $Q_{\mathrm{ref}}$ on $\Omega$). We can then define the divergence, or negative log-likelihood ratio,
\[\alpha_{\CQ|\bx_N}(Q) = -\ell(\bx_N;Q)+\sup_{Q\in\CQ}\ell(\bx_N;Q).\]

If we knew the `true' measure $Q$, a natural estimate for $\phi(X)$ given our observations $\bx_N$ would be the conditional expectation $E_Q[\phi(X)|\bx_N]$. Formally, as this is only defined for almost every observation $\bx_N$, and we have a possibly uncountable family of measures $\CQ$, this may cause technical difficulties. As we are working with countably generated Borel spaces\footnote{See, for example, Cohen and Elliott \cite[Section 2.6]{SCA} or Shiryaev \cite[Section II.7]{Shiryaev2000} for further details on the mathematics of this construction.}, simultaneously for every $Q\in \mathcal{Q}$, we can define the regular conditional $Q$-probability distribution of $X$ given $\bx_N$, that is, a probability kernel $\mathbf{Q}: \bR^N\times \mathcal{B}(\bR) \to [0,1]$ such that 
\[\mathbf{Q}(\bx_N(\omega), A) = Q(X\in A| \bx_N)(\omega),\]
and so obtain a `regular' conditional expectation
\[E_Q[\phi(X)|\bx_N] := \int_\bR \phi(x) \mathbf{Q}(\bx_N, dx).\]
We use this regular conditional expectation in what follows.

\begin{remark}
 We shall focus our attention on the simpler case where $\{X\}\cup\{X_n\}_{n \in \bN}$ are iid under each $Q\in\CQ$. In this case, each measure in $\mathcal{Q}$ can be described by the density $f_Q$ of $X$ under $Q$, and we obtain the simplifications $\ell(\bx_N;Q) = \sum_{i=1}^N \log f_Q(X_i)$ and $E[\phi(X)|\bx_N] = E[\phi(X)]$.
\end{remark}

The key idea of \cite{Cohen2016} is to define the following operator.
\begin{definition}
 For fixed constants $k>0$, $\gamma\in [1, \infty]$, define the nonlinear expectation
\[\CE^{k,\gamma}_{\CQ|\bx_N}(\phi(X)) = \sup_{Q\in\CQ}\Big\{ E[\phi(X)|\bx_N] - \Big(\frac{1}{k}\alpha_{\CQ|\bx_N}(Q)\Big)^\gamma\Big\},\]
where we write $|x|^\infty = 0$ for $|x|\leq 1$, and $\infty$ otherwise. We call this the DR-expectation (with parameters $k,\gamma$). `DR' is a deliberately ambiguous acronym, and refers to either `divergence robust' or `data-driven robust'.
\end{definition}

We can see that $\CE^{k,\gamma}_{\CQ|\bx_N}$ is built from considering the expected values under a range of models $\CQ$, penalized by how well each model fits the observed data $\bx_N$. This allows the DR-expectation to encode the uncertainty inherent in statistical estimation.

This nonlinear expectation has many elegant properties. From the perspective of the theory of nonlinear expectations, it is monotone, convex, constant equivariant and assigns all constants their values. One can obtain a risk measure, in the sense of F\"ollmer and Schied \cite{Follmer2002}, Frittelli and Rosazza-Gianin \cite{Frittelli2002} and Artzner, Delbaen, Eber and Heath \cite{Artzner1999}, by specifying $\rho(\phi(X)) = \CE^{k,\gamma}_{\CQ|\bx_N}(-\phi(X))$, or by treating $\phi(X)$ as `losses'. The nonlinear expectation is coherent (in particular, it is positively homogeneous) if $\gamma=\infty$. 

At the same time, unlike most risk measures in the literature, the DR-expectation is directly built from the observed data to represent statistical uncertainty. In order to perform classical statistical estimation, an agent needs to specify a class of models $\CQ$ to be considered (and in a Bayesian analysis, a prior distribution over models); the only remaining inputs to the problem are the observations $\bx_N$. In the DR-expectation framework, one additionally needs the uncertainty aversion parameter $k$ and the curvature parameter $\gamma$. Rather than separating statistical inference from the valuation of outcomes, the DR-expectation combines the two, as the likelihood is part of the definition of the expectation. Observations $\bx_N$ do not simply affect our expectation through conditioning (as in a Bayesian analysis), they also  determine which models $Q\in \CQ$ are `good', in the sense of fitting our past observations well.

\begin{remark}
A key advantage of this approach is that we define the DR-expectation for every function $\phi$. This allows the comparison of different $\phi$ to be consistently carried out, and the impact of statistical uncertainty on different $\phi$ may vary. While we have assumed above that $X$ is real-valued, there is no issue if $X$ is taken to be $\bR^d$-valued (or more generally, Borel measurable and valued in some separable Banach space), which allows a great deal of flexibility in modelling. 
\end{remark}

In the setting where $\{X\}\cup\{X_n\}_{n\in \bN}$ are iid, and under some regularity assumptions on $\CQ$ detailed in \cite{Cohen2016}, one can show that the DR-expectation is a consistent estimator of the expected value of $\phi(X)$, that is, 
\[\CE^{k,\gamma}_{\CQ|\bx_N}(\phi(X)) \to_P E_P[\phi(X)]\quad \text{as }N\to\infty, \text{ for every }P\in\CQ,\] 
 for any value of $k>0$ and $\gamma\in [1,\infty]$. It is also clear that the case $\gamma=\infty$ is closely related to the Neyman--Pearson lemma from hypothesis testing, as it considers models where the negative log-likelihood-ratio $\alpha_{\CQ|\bx_N}(Q)$ is sufficiently small. 

In order to give precise asymptotic statements, the following definition will prove useful. Here $P^*$ is the outer measure associated with $P$, to deal with any potential lack of measurability.
\begin{definition}
Consider sequences $f=\{f_n\}_{n\in \bN}$ and $g=\{g_n\}_{n\in \bN}$ of functions $\Omega \to \bR$.
\begin{enumerate}[(i)]
 \item We  write $f=O_P(g)$ whenever  $f_n/g_n$ is stochastically bounded, that is,  $P^*(|f_n/g_n|>M)\to 0$ as $M\to \infty$ for each $n$,
 \item We write $f=o_P(g)$ whenever $\lim_{n\to\infty}P^*(|f_n/g_n|>\epsilon)=0$ for all $\epsilon>0$.
\end{enumerate}
Note that this depends on the choice of measure $P$. 
\end{definition}

\begin{remark}
In \cite{Cohen2016}, a range of large-sample asymptotic results were obtained, for iid observations and $X$, under varying assumptions on the family $\CQ$. Suppose that $\CQ$ is a subset of a `nice' exponential family parametrized by $\theta\in \Theta$  (in particular, Assumption \ref{assn:expfam}, as detailed below, holds)  and the maximum likelihood estimator based on a sample of size $N$, denoted $\hat\theta_N$, is well defined. For bounded $\phi$, the following large sample approximations were obtained:
\begin{align*}
 \CE^{k,1}_{\CQ|\bx_N}(\phi(X)) &= E_{\hat \theta_N}[\phi(X)] + \frac{k}{2N} V(\phi, \hat\theta_N) + O_P(N^{-3/2}),\\
\CE^{k,\infty}_{\CQ|\bx_N}(\phi(X)) &= E_{\hat \theta_N}[\phi(X)] + \sqrt{\frac{2k}{N} V(\phi, \hat \theta_N)} + O_P(N^{-3/4}),
\end{align*}
for every $P\in \mathcal{Q}$, where $V(\phi, \hat \theta_N)$ is the `local variance' at the MLE, defined by 
\[V(\phi, \hat\theta) := \Big(\frac{\partial E_\theta[\phi(X)]}{\partial \theta}\Big|_{\hat \theta}\Big)^\top (\mathfrak{I}_{\hat\theta}^{-1})\Big(\frac{\partial E_\theta[\phi(X)]}{\partial \theta}\Big|_{\hat \theta}\Big),\]
 for $\mathfrak{I}_{\hat\theta_N}$ the observed information matrix (the Hessian of the negative log-likelihood) at $\hat\theta_N$.
\end{remark}

\begin{remark}
 The results in \cite{Cohen2016} focus on the case where $\{X\}\cup\{X_n\}_{n\in \bN}$ are iid under every model in $\CQ$. This is clearly restrictive from a modelling perspective. However, it is clear from the definition that the DR-expectation depends on the set $\CQ$  only through the conditional expectation and the log-likelihood, evaluated at the observed data $\bx_N$. Therefore, provided our models are `close' to an iid setting, in the sense that our expectations and likelihoods are similar to those from an iid model (possibly after some transformation of the space), the asymptotic behaviour of the DR-expectation should not be significantly affected.
\end{remark}

\begin{remark}
 The DR-expectation is particularly designed to highlight statistical uncertainty, rather than an individual's risk aversion. In particular, the above results (particularly when combined with the classic analysis of White \cite{White1982}) show that as the sample size $N\to\infty$, in the iid case,
\[\CE^{k,\gamma}_{\CQ|\bx_N}(\phi(X))\approx -\CE^{k,\gamma}_{\CQ|\bx_N}(-\phi(X)) \approx E_P[\phi(X)],\]
where $P$ is the measure in $\CQ$ which is closest to the empirical distribution of the data. In this way, an individual's preferences can be seen not to play any role in the DR-expectation when samples are large (that is, when statistical uncertainty is low).
\end{remark}

\subsection{Data-robust risk measures}
To focus attention on extreme events, or to include an individual's risk aversion, we may wish to extend our focus from the DR-expectation. Classically, when we ignore statistical uncertainty (so the DR-expectation would be replaced by the expectation with respect to the known distribution), a standard approach would be to take $X$ to denote losses and $\phi$ to be a utility function. In our setting, following this approach would lead to a version of expected utility theory where the expectation is replaced with the DR-expectation.

Further interesting cases can be obtained by directly combining the DR-expectation with a risk assessment depending directly on the law of the random outcome. As we shall see, many common risk assessments are of this type, for example the value at risk and expected shortfall, as well as classical statistics such as the expectation and variance. Writing $\CM_1$ for the space of probability measures on $\bR$, we consider a map $\CR:\CM_1\to \bR$ which represents the `riskiness' of a gamble with specified distribution. We can then combine the risk aversion of $\CR$ (which treats the law of $X$ as fixed) with the uncertainty aversion of the DR-expectation (which considers our uncertainty in this law). We formalize this construction in the following definition.

\begin{definition}
Write $\CM_1$ for the space of probability measures on $\bR$ and $\CL_Q(\xi|\bx)$ for the (regular conditional) law of a random variable $\xi$ under the measure $Q$ (given observations $\bx$). Let $\CR:\CM_1\to \bR$ be a map with $\CR(\CL_Q(\xi|\bx))$ representing the risk of our position $\xi$ assuming $Q\in \CQ$ is a `true' model. We combine $\CE^{k,\gamma}_{\CQ|\bx}$ and $\CR$ through the definition
\begin{align*}
(\CE^{k,\gamma}_{\CQ|\bx}\circ \CR)(\xi) & := \sup_{Q\in\CQ}\Big\{\CR(\CL_Q(\xi|\bx))-\Big(\frac{1}{k}\alpha_{\CQ|\bx}(Q)\Big)^\gamma\Big\}
\end{align*}
\end{definition}
\begin{remark}
 For the sake of simplicity, we will identify this construction by adding a prefix `DR-'. For example, combining the DR-expectation and the expected shortfall, as in the example below, we obtain the `DR-expected shortfall', similarly the `DR-Value at Risk'.
\end{remark}

\begin{example}\label{ex:ESdefn}
Consider the risk assessment given by the (upper) expected shortfall at level $\epsilon$ (also called the conditional value at risk, tail value at risk or average value at risk, by various authors). For a model $P$ such that $\xi$ has a continuous distribution, this can be written
\[\CR(\CL_P(\xi)) = E_P[\xi|\xi\geq F^{-1}_P(1-\epsilon)]=\ES_\epsilon(\xi)\]
where $F_P$ is the cdf of $X$ under the measure $P$. This is a coherent risk measure (or nonlinear expectation), and can also be written (see F\"ollmer and Schied \cite{Follmer2002} or further discussion in McNeil, Frey and Embrechts \cite{McNeil2005})
\[\ES_\epsilon(\xi) = \sup_{P'\sim P}\Big\{ E_{P'}[\xi] -\tilde\alpha(P';P)\Big\}\]
where 
\[\tilde\alpha(P';P) = \begin{cases} 0 & \text{if }\|dP'/dP\|_\infty <\epsilon^{-1},\\ \infty & \text{otherwise.}\end{cases}\]
\end{example}

In the special case where $\CR$ corresponds to a convex expectation (based on a reference measure), this construction can simplify in a natural way. 

\begin{proposition}\label{expconvolution}
Suppose that, for each measure $P\in\CQ$, our risk assessment $\CR$ has representation
\[\CR(\CL_P(\xi|\bx))= \sup_{P'\in \tilde\CQ}\{E_{P'}[\xi|\bx]-\tilde\alpha(P';P)\},\]
where $\tilde\CQ$ is a collection of probability measures on $\Omega$ with $\CQ\subseteq\tilde\CQ$ and $\tilde\alpha:\tilde\CQ\times\tilde\CQ\to \bR$ is an arbitrary penalty function. The composition of $\CE^{k,\gamma}_{\CQ|\bx}$ and $\CR$ is a nonlinear expectation, and has representation
\begin{align*}
(\CE^{k,\gamma}_{\CQ|\bx}\circ \CR)(\xi)  = \sup_{P'\in \tilde\CQ}\big\{E_{P'}[\xi|\bx]-\alpha^*(P')\big\}
\end{align*}
where $\alpha^*$ is the inf-sum (cf. the inf-convolution in Barrieu and El Karoui \cite{Barrieu2005})
\[\alpha^*(P') = \inf_{Q\in\CQ}\Big\{\Big(\frac{1}{k}\alpha_{\CQ|\bx}(Q)\Big)^\gamma+\tilde\alpha(P';Q)\Big\}.\]
\end{proposition}
\begin{proof}
To obtain the representation, simply expand
\begin{align*}
(\CE^{k,\gamma}_{\CQ|\bx}\circ \CR)(\xi) & = \sup_{Q\in\CQ}\Big\{\CR(\CL_Q(\xi|\bx))-\Big(\frac{1}{k}\alpha_{\CQ|\bx}(Q)\Big)^\gamma\Big\}\\
& = \sup_{Q\in\CQ, P'\in \tilde\CQ}\Big\{E_{P'}[\xi|\bx]-\Big(\frac{1}{k}\alpha_{\CQ|\bx}(Q)\Big)^\gamma-\tilde\alpha(P';Q)\Big\}\\
& = \sup_{P'\in \tilde\CQ}\big\{E_{P'}[\xi|\bx]-\alpha^*(P')\big\}
\end{align*}
This defines a nonlinear expectation by duality, as in F\"ollmer and Schied \cite{Follmer2002}.
\end{proof}

\section{Regularization from data}

 The above approach, based on simply penalizing with the divergence, often fails to give bounded values when considering unbounded random variables. In some sense, this is because of the difficulty in using data to determine the probabilities of extreme outcomes. To motivate our discussion, we consider the following simple example.

\begin{example}
 Consider the case where $\xi=\phi(X)=\beta X$, and $\{X\}\cup\{X_n\}_{n\in\bN}$ are iid $N(\mu,\sigma^2)$ distributed random variables, with $\mu\in\bR$ and $\sigma^2>0$  unknown. Writing $\bar X = N^{-1}\sum_{i=1}^N X_i$ and $\hat\sigma^2 = N^{-1}\sum_i (X_i-\bar X)^2$ for the MLE estimates of $\mu$ and $\sigma^2$, it is straightforward to calculate (see \cite{Cohen2016} for details)
\begin{equation}\label{normDRdecomp}
 \begin{split}
\CE_{\mathcal{Q}|\mathbf{x}_N}^{k,1}(\beta X)&=\sup_{\mu,\sigma^2}\Big\{\beta\mu - \frac{N}{2k}\Big(\log(\sigma^2/\hat\sigma^2)+\frac{\frac{1}{N}\sum_{n=1}^N (X_n-\mu)^2}{\sigma^2}-1\Big)\Big\}\\
&=\sup_{\mu,\sigma^2}\Big\{\beta\mu -\frac{N}{2k\sigma^2}(\bar X-\mu)^2 -\frac{N}{2k}\Big(\log(\sigma^2/\hat\sigma^2)+\frac{\hat\sigma^2}{\sigma^2}-1\Big)\Big\}\\
&= \beta \bar X +\sup_{\sigma^2} \Big\{ \frac{\beta^2 k}{2N}\sigma^2 - \frac{N}{2k}\Big(\log(\sigma^2/\hat\sigma^2)-1+\frac{\hat\sigma^2}{\sigma^2}\Big)\Big\}.
\end{split}
\end{equation}
This is problematic, as for any finite $N$, the growth of $\sigma^2$ will dominate the term in braces as $\sigma^2\to\infty$, which implies $\CE_{\mathcal{Q}|\mathbf{x}_N}^{k,1}(\beta X)=\infty$. 

However, the function to be optimized is increasing near zero, and we can find a local minimum value of the derivative at $\sigma^2=2\hat\sigma^2$, where
\[\frac{\partial}{\partial \sigma^2} \bigg(\frac{\beta^2 k}{2N}\sigma^2 - \frac{N}{2k}\Big(\log(\sigma^2/\hat\sigma^2)-1+\frac{\hat\sigma^2}{\sigma^2}\Big)\bigg)\bigg|_{\sigma^2=2\hat\sigma^2} = \frac{\beta^2 k}{2N} - \frac{N}{2k}\frac{1}{4\hat\sigma^2}.\]
Provided this quantity is negative, that is, $\hat\sigma^2 > (\beta k/2N)^2$, we know that the derivative has changed sign at least once on the interval $\sigma^2\in (0, 2\hat\sigma^2)$, so we can be sure that the function will have a local maximum, and (we shall see that) this occurs near $\sigma^2\approx \hat\sigma^2$.

This example leads us to consider further ways of restricting our attention, to ensure that it is a \emph{local} maximum which is chosen in the DR-expectation, as this corresponds to a value close to the MLE, and avoids the explosion of the expectation caused by considering the implausible model $\sigma^2\to \infty$.
\end{example}

The following are some natural approaches to preventing this explosion:
\begin{itemize}
 \item We could place a prior probability on $\sigma^2$, with density $f$, which collapses sufficiently quickly when $\sigma^2\to \infty$, and then incorporate this prior into the penalty. This would result in a penalty given by the negative log-posterior density,
\[\alpha_{\CQ|\bx_N}(Q) = -\ell(\bx_N; Q) - \log f(\sigma^2_Q) + \text{normalization term}\]
where the normalizing term is constant with respect to $Q$, and ensures $\inf_{Q\in\CQ} \alpha_{\CQ|\bx_N}(Q)=0$. In order to ensure $\CE_{\mathcal{Q}|\mathbf{x}_N}^{k,1}(\beta X)<\infty$ for large $N$, we require $f(\sigma^2)e^{a\sigma^2}\to 0$ as $\sigma^2\to \infty$, for some $a>0$, which is not the case for the classical (inverse Gamma) conjugate prior. This also seems a rather \emph{ad hoc} fix, and raises the concern that the conclusions drawn, no matter how much data is available, depends principally on the choice of prior distribution $f(\sigma^2)$. 
 \item We could \emph{a priori} truncate $\sigma^2$ away from $\infty$ (i.e. set a maximum value $\bar{\sigma}^2$ which will be considered). This is equivalent to a uniform prior $f(\cdot)\propto I_{[0,\bar{\sigma}^2]}$ for $\sigma^2$, and will result in the local maximum being chosen whenever $N$ is sufficiently large. However this has the drawback that it requires us to posit a maximum value of $\sigma^2$ independently of the data, and again our conclusions will depend on the choice of this upper bound, for all finite values of $N$.
 \item We could restrict to the case $\gamma=\infty$, and thus only consider a likelihood interval of values for $\sigma^2$. This has the implication that our DR-expectation is then positively homogeneous, and we `flatten out' the interaction between the choice of model and the variable being considered. (In particular, the penalty term only takes the values $0$ and $\infty$ -- models are either excluded or are considered as reasonable as the best-fitting model, with no middle ground.)
\end{itemize}

The final of these options has the particular advantage that the conclusions are purely based on the data, rather than our prior assumptions. At the same time, the flattening of our expectation may be undesirable, as it makes it difficult to see how different models will be chosen when evaluating different random variables. For this reason, we propose an alternative regularization technique, combining the $\gamma<\infty$ and $\gamma=\infty$ cases. To achieve this, we define the following transformation of the divergence.

\begin{definition}
 The $\delta$-truncation of the divergence, based on a sample of size $N$, is given by 
\[\alpha^{(\delta)}_{\CQ|\bx_N}(Q) = \begin{cases} \alpha_{\CQ|\bx_N}(Q) & \text{if }\alpha_{\CQ|\bx_N}(Q) \leq N^\delta,\\
                                     \infty & \text{otherwise}.
                                    \end{cases}
\]
The corresponding DR-risk-assessment is defined by 
\begin{equation}\label{eq:truncoptim}
 (\CE^{k,\gamma,\delta}_{\CQ|\bx_N}\circ\CR)(\xi) = \sup_{Q\in\CQ}\Big\{\CR(\CL_Q(\xi)|\bx_N) - \Big(\frac{1}{k}\alpha^{(\delta)}_{\CQ|\bx_N}(Q)\Big)^\gamma\Big\}.
\end{equation}
\end{definition}

\begin{remark}
 For bounded random variables $\xi$, taking $\CR(\CL_Q(\xi|\bx))=E_Q[\xi|\bx_N]$, the optimization step in calculating the DR-expectation implies that only those measures in $\{Q:\alpha_{\CQ|\bx_N}(Q) < \|\xi\|_{\infty}\}$ need be considered. Therefore, provided $N^\delta>\|\xi\|_\infty$, we have $\CE^{k,\gamma,\delta}_{\CQ|\bx_N}(\xi) = \CE^{k,\gamma}_{\CQ|\bx_N}(\xi)$.
\end{remark}

\begin{remark}
 Suppose $\mathcal{Q}$ is sufficiently regular and $N$ sufficiently large that Wilks' theorem holds\footnote{This classical theorem states that $\alpha_{\CQ|\bx_N}(P)$ has an asymptotic $\chi^2_d$ distribution under $P$, where $d$ is the dimension of a parameter space representing $\CQ$. Conditions under which this holds (which can be interpreted as requirements on the family $\CQ$) can be found in Lehmann and Casella \cite{Lehmann1998}.}. Then the set of measures $\{Q:\alpha_{\CQ|\bx_N}(Q)<N^\delta\}$ corresponds to a confidence set with confidence level $F^{-1}_{\chi^2_d}(2N^\delta)$, for $d$ the dimension of (a parametrization of) $\mathcal{Q}$, where $F^{-1}_{\chi^2_d}$ is the inverse cdf of the $\chi^2_d$ distribution. In this case, we observe that the $\delta$ truncation corresponds to considering only those measures in a confidence set around the MLE (and then penalizing further using their divergences based on the data). However, the confidence level of this set will grow quickly with $N$ (when $d=1$, already for $N^\delta \approx 2$ we are considering a $95\%$ confidence set).
\end{remark}

We now consider the behaviour of a sequence $\{\CR_N\}_{N\in\bN}$ of risk assessments, which may vary with our sample size $N$.

\begin{definition}\label{def:regular}
 We say that $(\CE^{k,\gamma,\delta}_{\CQ|\bx_N}\circ \CR_N)(\xi)$ is regular whenever the supremum in its definition is finite and attained at a point where $\alpha_{\CQ|\bx_N}(Q) < N^\delta$. 
\end{definition}
This definition gives us a criterion, which can be evaluated for a particular set of observations, to determine whether the estimation of our truncated DR-risk assessment is reliable. If our DR-risk assessment is not regular, then our estimate is based on parameters on the boundary of the admissible region $\{Q:\alpha_{\CQ|\bx_N}(Q) < N^\delta\}$. This suggests our estimate  depends critically on the choice of truncation parameter $\delta$. On the other hand, when our DR-risk assessment is regular, we know that we have chosen a local maximum within the admissible region, which suggests that the variation in the truncation parameter $\delta$ has no marginal impact on our estimates.

The fact that Definition \ref{def:regular} depends on the particular observations $\bx_N$ is important, as it provides us with a qualitative criterion to assess reliability of estimates which can be calculated from our observations. At the same time, it is interesting to know whether there exists a choice of truncation parameter $\delta$ which will typically result in a regular DR-risk assessment.

\begin{definition}
We say $(\CE^{k,\gamma}_{\CQ|\bx_N}\circ \CR_N)(\xi)$ is ($\CQ$-)regularizable if, for some $\delta>0$, for all $P\in\mathcal{Q}$, with $P$-probability tending to $1$, as $N\to \infty$, we have that $(\CE^{k,\gamma,\delta}_{\CQ|\bx_N}\circ \CR_N)(\xi)$ is regular.  
\end{definition}

\begin{remark}
 Typically, the divergence will grow like $O(N)$ outside of any neighbourhood of the MLE. Therefore, regularizing with $\delta<1$ is sufficient to guarantee that only points in a neighbourhood of the MLE are considered in the optimization. The question is then whether, for risk assessments $\CR_N$, we can choose  $\delta$ sufficiently large that the local maximum is strictly in the interior of the restricted neighbourhood. If this can be done, then the problem is regularizable. This leads to a delicate interplay between the class of risk assessments, the sample size, and the estimation problem itself.
\end{remark}

\begin{example}\label{normexamplereg1}
 In the normal example above, $\CE^{k,1,\delta}_{\CQ|\bx_N}(\beta X)$ is regular whenever $\beta =o(N^{(\delta+1)/2})$ and $\delta<1$ (for large $N$, with $P$-probability approaching $1$ for every $P\in\CQ$). Consequently, $\CE^{k,1}_{\CQ|\bx_N}(\beta X)$ is regularizable for any fixed $\beta$. This is suggested by the following sketch argument (and will be rigorously proven later).

As $\delta<1$, we are restricting to a neighbourhood of the MLE, so we can be confident that we will select a local extremum. We need only to verify that this occurs at a point with penalty $\alpha_{\CQ|\bx_N}(\mu_*, \sigma^2_*)<N^\delta$. For large $N$, we approximate $\alpha_{\CQ|\bx_N}$ by a quadratic in $(\mu, \sigma^2)$. This gives us the approximate optimization problem 
\[\CE^{k,1,\delta}_{\CQ|\bx_N}(\beta X)\approx \sup_{\mu,\sigma^2} \Bigg\{\beta\mu -\bigg\{\frac{N}{2k\sigma^2}(\bar X-\mu)^2 +\Big(\frac{N}{4k}\Big)\Big(\frac{\sigma^2}{\hat\sigma^2}-1\Big)^2\bigg\}\Bigg\}.\]
The extremum of this approximation is attained by
\[\mu_* = \bar X+\frac{k}{N}\beta\sigma^2,\qquad \sigma^2_* = \hat\sigma^2\Big(1 + \frac{\beta^2 k^2}{N^2}\hat\sigma^2\Big).\]
We now substitute back into the quadratic approximation, to obtain
\begin{align*}
\alpha_{\CQ|\bx_N}(\mu_*,\sigma^2_*)&\approx\frac{k\beta^2}{2N}\hat\sigma^2\Big(1+ \frac{\beta^2k^2}{N^2}\hat\sigma^2\Big)+ \Big(\frac{N}{4k}\Big)\Big(\frac{\beta^2k^2}{N^2}\hat\sigma^2\Big)^2\\
 &=O_P(N^{-1}\beta^2 + N^{-3} \beta^4).
\end{align*}
Provided $\beta = o(N^{(\delta+1)/2})$, this guarantees $\alpha_{\CQ|\bx_N}(\mu_*,\sigma^2_*) = o_P(N^\delta)$, so our DR-expectation is regular. 

For fixed $\beta$ (or more generally $\beta = o(N^{1-\eta})$ for some $\eta>0$), we can find a $\delta$ satisfying these conditions. Therefore, with $P$-probability approaching one for every $P\in \CQ$, as $N\to \infty$, the DR-expectation is regular. Therefore, the DR-expectation is regularizable. 
\end{example}

 To make the argument above rigorous and more generally applicable, we proceed under the following assumption on our observations.
%
 \begin{assumption}\label{assn:expfam}
Assume that
\begin{enumerate}[(i)]
 \item The family $\mathcal{Q}$ is parametrized by an open set $\Theta\subseteq \bR^d$.
 \item Under each $Q\in\CQ$, the observations $\{X\}\cup\{X_n\}_{n\in \bN}$ are iid with density of the form
\[f(x;Q) = h(x)\exp\Big\{ \langle \theta, T(x)\rangle -A(\theta)\Big\}.\]
(in other words, $\mathcal{Q}$ is an exponential family with its natural parametrization), where $T$ and $A$ are fixed functions (called respectively the sufficient statistic and the log-partition function).
 \item The variable of interest is $\xi=\phi(X)$, for some Borel function $\phi$.
 \item The Hessian $\mathfrak{I}_{\theta} = \partial^2 A(\theta)$ (commonly known as the information matrix) is strictly positive definite at every point of $\Theta$. 
 \item The $\mathcal{Q}$-MLE exists and is consistent, with probability tending to $1$ as $N\to\infty$ (that is, for every $P\in \mathcal{Q}$, a maximizer $\hat Q_N$ of the likelihood in $\CQ$ exists with $P$-probability approaching $1$ and $\hat\theta_N=\theta_{\hat Q_N} \to_P \theta_P$).
\end{enumerate}
\end{assumption}

The above assumption leads to the following estimate, which gives us a good understanding of the asymptotic behaviour of $\alpha_{\CQ|\bx_N}$. For simplicity of notation, we shall write $\alpha_{\CQ|\bx_N}(\theta)$ for $\alpha_{\CQ|\bx_N}(Q_\theta)$, where $Q_\theta$ is the measure parametrized by $\theta\in\Theta$.
\begin{lemma}\label{quadpenbound}
 Suppose Assumption \ref{assn:expfam} holds. Then for every $P\in \CQ$, there exists a constant $C>0$ depending on $P$ but independent of $N$, such that we have the uniform bound
\[P\Big(\alpha_{\CQ|\bx_N}(\theta)\geq \frac{N}{C}\|\theta-\hat\theta_N\|(1\wedge \|\theta-\hat\theta_N\|)\text{ for all }\theta\Big) \to 1.\]
\end{lemma}
\begin{proof}
 This follows from the proof of Lemma 3 in \cite{Cohen2016}.
\end{proof}

The following abstract approximation result will be useful.
\begin{theorem}\label{thm:approxthm}
Consider the maximization of the general function 
\[f_N(\psi) := g_N(\psi) - \frac{1}{k}\alpha_{\CQ|\bx_N}(\hat\theta+\psi)\]
Suppose that $\alpha_{\CQ|\bx_N}$ arises from a setting where the result of Lemma \ref{quadpenbound} can be applied and, for some $\epsilon\in\,]0,1/2]$, for every $P\in\CQ$, for every $N$ the function $g_N$ is $C^3$ and satisfies, for $B$ a ball of constant radius around $\psi=0$,
\begin{enumerate}[(i)]
\item $g_N(0)=0$,
 \item  $\|g_N'(\psi)\|=o_P(N^{1-\epsilon})$ uniformly on $B$,
  \item $\|g_N''(0)\| = o_P(N)$,
  \item $\|g_N'''(\psi)\| = O_P(N)$ uniformly on $B$.
\end{enumerate}
Then defining
\[\delta=1-2\epsilon\]
there exists a point $\psi^*_N$ such that, for each $P\in\CQ$, with $P$-probability approaching $1$ as $N\to\infty$, we know that $\psi^*_N$ maximizes the value of $f_N$ on the set  $\{\psi:\alpha_{\CQ|\bx_N}(\hat\theta+\psi)\leq N^{\delta}\}$ and, furthermore, $\alpha_{\CQ|\bx_N}(\hat\theta+\psi^*_N)= o_P(N^{\delta})$.
\end{theorem}
\begin{proof}
 See Appendix.
\end{proof}

\begin{example}
In the normal setting considered in Example \ref{normexamplereg1}, we are seeking to maximize
\[\beta\mu - \frac{1}{k}\alpha_{\CQ|\bx}(\mu, \sigma^2) = \beta \bar X +\big\{ \beta \psi_\mu  - \frac{1}{k}\alpha_{\CQ|\bx}(\psi)\big\} \]
where $\psi= (\psi_\mu, \psi_{\sigma^2}) = (\mu-\bar X, \sigma^2-\hat\sigma^2)$. Provided $\beta = o(N^{1-\eta})$ for some $\eta>0$, we can take $g(\psi) = \beta \psi_\mu$, which satisfies the conditions of Theorem \ref{thm:approxthm}. By application of the theorem, it follows that the DR-expectation is regularizable in this case, giving a rigorous proof of the condition obtained by quadratic approximation in Example \ref{normexamplereg1}.
\end{example}

 In the above example, we have considered the case $\gamma=1$. We can also obtain a sufficient condition for the case $\gamma>1$ using an extension of this method, assuming that $\epsilon= 1/2$.
\begin{lemma} \label{lem:delta0reg}
 Suppose that, in the case $\gamma=1$, our truncated optimization problem \eqref{eq:truncoptim} has a maximizer $Q^*_1$ where $\alpha_{\CQ|\bx_N}(Q^*_1)<1$, when taking a truncation parameter $\delta\geq 0$. Then the truncated problem in the case $\gamma\geq1$ certainly has a maximizer $Q^*_\gamma$ with $\alpha_{\CQ|\bx_N}(Q^*_\gamma)<1$, using the same truncation parameter.
\end{lemma}
\begin{proof}
 This is immediate from the fact that $|x|^\gamma<|x|$ for $|x|<1$ and $|x|^\gamma>|x|$ for $|x|>1$.
\end{proof}

\begin{example}
In the normal setting, we see that if $\beta=o(N^{1/2})$ then Theorem \ref{thm:approxthm} yields the regularization parameter $\delta=1-2\epsilon =0$. In other words, we truncate to a set $\alpha_{\CQ|\bx_N}< o_P(1)$. By Lemma \ref{lem:delta0reg}, this implies that the same regularization will be sufficient for all choices of $\gamma<\infty$.

This argument suggests that $\beta=o(N^{1/2})$ is a fundamental requirement for reliable statistical estimation of $E[\beta X]$ from $N$ observations, under the assumption that $X$ and our observations are iid normal with unknown mean and variance. Comparing with the case $\gamma=\infty$, where the DR-expectation is the upper bound of a confidence interval, we would need $\beta=o(N^{1/2})$ to ensure that 
\[\big|\CE^{k,\infty}_{\CQ|\bx_N}(\beta X) - E_{\hat Q}[\beta X]\big| \approx \sqrt{2k}\frac{|\beta| }{\sqrt{N}}\hat\sigma \to 0,\]
 where $\hat Q$ is the MLE.
\end{example}

\begin{remark}
 The example above shows that there is a close asymptotic relationship between regularity and the convergence of confidence intervals. However, the concept of regularity has the advantage that it can be applied for a given sample (of fixed size), rather than only being meaningful in an asymptotic sense. In particular, for a given sample, we can evaluate (for fixed $\delta\geq0$) whether our truncated DR-expectation is regular, which allows us to qualitatively assess the statistical reliability of our estimates \emph{based on the data at hand, and for the variables of interest}. In particular, if the function to be maximized when calculating $(\CE^{k, \gamma}_{\CQ|\bx_N}\circ \CR)(\xi)$ does not have a local maximum near the MLE, then we know that the DR-risk assessment is not regular for any $\delta$.  This is unlike a confidence interval, which will generally not give a qualitative assessment of reliability for any fixed sample.

If we find that our problem is not regular, then we know that the observed data are insufficient to guide our calculation of DR-expectations -- our choice of regularization has an overriding impact on our conclusions.
\end{remark}

\section{Heavy tails}
The example given in the previous section shows us that, in a Gaussian setting, provided $\beta$ is not too large relative to $N$ (in particular $\beta=o(N^{1-\eta})$ for some $\eta>0$),  the DR-expectation $\CE^{k,1}_{\CQ|\bx_N}(\beta X)$ is regularizable. This suggests asymptotic bounds on the size of random variables we are willing to consider, if we wish to account for our statistical uncertainty. In what follows, we shall seek to show that a similar relationship exists in a heavy tailed setting, when we calculate expected shortfall and related quantities.

We shall focus our attention on the following special case.

\begin{assumption}\label{assn:Pareto}
The family of measures $\CQ$ describes models under which our observations $\{X\}\cup\{X_n\}_{n\in\bN}$ are iid from a Pareto distribution (with known minimal value $1$), that is, for a model parametrized by $\theta$, we have the density
\[f(x;\theta) = \theta x^{-(1+\theta)}; x>1.\]
The possible values of $\theta$ will be assumed to satisfy one of two cases:
\begin{enumerate}[(i)]
 \item $\CQ$ corresponds to all $\theta>0$ (which implies a valid probability distribution for $X$, but no integrability),
 \item $\CQ$ corresponds to all $\theta>1$ (which implies $X$ is integrable).
\end{enumerate}
\end{assumption}
We should note that this is an exponential family of distributions, so the above approximation results (in particular Lemma \ref{quadpenbound}) can be applied. Clearly Assumption \ref{assn:Pareto}(ii) is stronger than Assumption \ref{assn:Pareto}(i).

It is an easy exercise to show the following:
\begin{proposition}
In the setting of Assumption \ref{assn:Pareto}, the MLE is given by 
\[\hat\theta = \frac{N}{\sum_i \log(x_i)}\]
and the divergence by
\[\alpha_{\CQ|\bx_N}(\theta) = N\bigg(-\log\Big(\frac{\theta}{\hat\theta}\Big) -1+\frac{\theta}{\hat\theta}\bigg).\]
\end{proposition}

\begin{remark}
This is a significantly simplified version of a standard extreme value estimation problem. In particular, consider a model for excesses over a threshold $u$. By the Pickands--Balkema--de Haan theorem (see for example Embrechts, Kl\"uppelberg and Mikosch \cite{Embrechts1997}), we know that for large $u$, the probability distribution in such a setting is typically approximated by the generalized Pareto distribution
\[P(X-u\leq y|X>u) \approx \begin{cases} 1-\Big(1+\frac{ky}{\sigma}\Big)^{-1/k} & \text{if }k\neq 0,\\
                                     1-e^{-y/\sigma}& \text{if }k=0.
                                    \end{cases}\]
If we assume that this is the correct distribution family (so we are in the limiting regime), the scale variable $\sigma$ is known (for simplicity) and $k>0$ (so our distribution is unbounded and not exponential), then estimating $k$ is equivalent to estimating $\theta$ in the above (classical) Pareto setting (after reparametrization). 

Of course, in practice, the estimation of $\sigma$ and the choice of cut-off value $u$ are significant issues in their own right. What we shall see is that, even without these additional concerns, the estimation error associated with $k$ (equivalently $\theta$) is sufficient to restrict which risk assessments can be reliably estimated.
\end{remark}

In what follows, we shall give relationships, for a variety of common risk assessment methods, between the growth of a parameter $\beta$ (which will describe the `extremity' of our risk assessment) and the sample size $N$, such that the DR-risk assessment is regularizable. This gives a method of assessing the reliability of our statistical estimates. We should point out that, if this were to be used in the context of fitting the excesses over a threshold, our sample size $N$ would refer exclusively to the number of excesses we observe, and any probabilities would be conditional on being over the threshold point.

\subsection{Expected shortfall}
The first example we shall consider is the DR-expected shortfall, as defined in Example \ref{ex:ESdefn}. As expected shortfall is a convex expectation, we know that the DR-expected shortfall is also a convex expectation. Given the expected shortfall is only defined for integrable random variables, we shall work under Assumption \ref{assn:Pareto}(ii).

\begin{lemma}
 Without truncation, under Assumption \ref{assn:Pareto}(ii), the DR-expected shortfall is infinite for all $N$, and all probability levels $\beta$, that is 
\[(\CE^{k,\gamma}_{\CQ|\bx_N}\circ \mathrm{ES}_\beta)(X) = \infty\]
for every set of observations $\bx_N$ and all choices of $k,\gamma<\infty$.
\end{lemma}
\begin{proof}
 Under Assumption \ref{assn:Pareto}, for a given $\theta>1$, we can calculate the Expected shortfall with tail probability $\beta$,
\[\mathrm{ES}_\beta = E[X|X>F_{\theta}^{-1}(1-\beta)] = \frac{\theta}{\theta-1}\beta^{-1/\theta}.\]
As $\theta\downarrow 1$, we observe that $\mathrm{ES}_\beta \to \infty$, but $\alpha_{\CQ|\bx_N}(\theta)\not\to \infty$.   Therefore, for every $\gamma<\infty$, it is easy to see that $(\CE^{k,\gamma}_{\CQ|\bx_N} \circ \mathrm{ES}_\beta)(X) =\infty.$
\end{proof}
\begin{remark}
 In the case $\gamma=\infty$, by monotonicity we will have 
\[(\CE^{k,\gamma}_{\CQ|\bx_N} \circ \mathrm{ES}_\beta)(X) = \frac{\theta^*}{\theta^*-1}\beta^{-1/\theta^*},\]
where $\theta^*$ is the smallest value\footnote{This assumes $\theta^*>1$, otherwise our risk assessment takes the value $\infty$.} such that $\alpha_{\CQ|\bx_N}(\theta)\leq k$. This condition is the same for all $N$ -- in particular, using only a single observation, we  give a finite estimate of the expected shortfall at any level $\beta\approx 0$ (i.e. arbitrarily far into the tail), which is intuitively unreasonable.  This is the `confidence interval' bound discussed in Embrechts, Kl\"uppelberg and Mikosch \cite{Embrechts1997}.
\end{remark}

This problem can be treated through regularization techniques, giving qualitatively different conclusions.
\begin{theorem}
 Under Assumption \ref{assn:Pareto}(ii), in the case $\gamma=1$, the DR-expected shortfall with tail probability $\beta$ is regularizable whenever $\beta^{-1}= o(N^{1-\eta})$, for some $\eta>0$.
\end{theorem}
\begin{proof}
 As the MLE is consistent, we know that for large $N$, $\hat\theta>1$ with $P$-probability approaching $1$, for every $P\in\CQ$. In what follows, we assume this is the case.

The quantity to be (locally) maximized when calculating \eqref{eq:truncoptim} is 
\[\frac{\theta}{\theta-1} \beta^{-1/\theta} - \frac{N}{k}\Big(\log(\hat \theta/\theta) -1+ \theta/\hat \theta\Big).\]
To apply Theorem \ref{thm:approxthm}, as we are interested in small values of $\beta$, we calculate the derivatives
\begin{align*}
 \frac{d}{d\theta}\Big(\frac{\theta}{\theta-1} \beta^{-1/\theta}\Big)&=\beta^{-1/\theta}\Big(\frac{\log\beta}{\theta(\theta-1)} -\frac{1}{(\theta-1)^2}\Big),\\
 \frac{d^2}{d\theta^2}\Big(\frac{\theta}{\theta-1} \beta^{-1/\theta}\Big)&=\beta^{-1/\theta}\Big(\frac{2}{(\theta-1)^3} -\frac{2\log\beta}{\theta(\theta-1)^2} +\frac{(\log\beta)^2}{\theta^3(\theta-1)}\Big),\\
 \frac{d^3}{d\theta^3}\Big(\frac{\theta}{\theta-1} \beta^{-1/\theta}\Big)&=O(\beta^{-1/\theta} (\log \beta)^3) \quad \text{uniformly near }\theta> 1.
\end{align*}
For $N$ large, uniformly in a neighbourhood of $\hat\theta>1$, taking 
\[g_N(\psi) = \frac{\hat\theta+\psi}{\hat\theta+\psi-1} \beta^{-1/(\hat\theta+\psi)}- \frac{\hat\theta}{\hat\theta-1} \beta^{-1/\hat\theta}\]
the requirements of Theorem \ref{thm:approxthm} are guaranteed by the simple condition
\[\frac{1}\beta = o_P(N^{\hat\theta-\eta})\qquad \text{for some }\eta>0, \text{ for all }P\in\CQ.\]
As $\hat\theta>1$ with $P$-probability approaching $1$ for $P\in \CQ$ (by consistency of the MLE), we obtain the deterministic bound desired.
\end{proof}

\begin{remark}
 This result places a qualitative bound on the extremity of the expected shortfall that can be reliably estimated, when we measure uncertainty through the DR-expected shortfall with $\gamma=1$. Given $N$ observations, for Pareto distributed models assuming only integrability, we cannot generally claim to estimate expected shortfalls for probabilities below $c/N$, where $c$ is some constant. In particular, estimating extreme tails from small numbers of observations is shown to be unreliable.
\end{remark}
\begin{corollary}
 Under Assumption \ref{assn:Pareto}(ii), in the case $\gamma<\infty$, the DR-expected shortfall with tail probability $\beta$ is regularizable whenever $\beta^{-1} = o_P(N^{1/2-\eta})$, in particular when $\beta^{-1}= o(N^{1/2-\eta})$, for some $\eta>0$.
\end{corollary}
\begin{proof}
By assumption, we know that $\beta^{-1}=o_P(N^{\hat \theta/2-\eta})$, we can apply Theorem \ref{thm:approxthm} with $\epsilon = 1/2$. This gives us a regularization parameter $\delta=0$, and from Lemma \ref{lem:delta0reg} we observe regularity for all $\gamma<\infty$.
\end{proof}

\begin{remark}
 In practice, this is still a very optimistic requirement for reliable estimation, as we have assumed that our simple Pareto model is correct (and no further parameters need to be estimated). In this sense,  these results give `best-case' bounds on how far into the tail we can look before losing reliability of expected shortfall estimation using a Pareto model. 
\end{remark}

\subsection{Value at Risk}
The Value at Risk (with tail probability $\beta$) is not a convex risk measure, so it is not generally true that $\CE_{\CQ|\bx_N}^{k, \gamma}\circ \mathrm{VaR}_\beta$ is a convex expectation. Nevertheless,  we can calculate
\[\mathrm{VaR}_\beta(X) = F_\theta^{-1}(1-\beta)= \beta^{-1/\theta}.\]
Given this is well defined (and finite) for every distribution, we shall proceed under Assumption \ref{assn:Pareto}(i).
\begin{lemma}
 Without truncation, under Assumption \ref{assn:Pareto}(i), the DR-value at risk is infinite for all $N$, and all probability levels $\beta<1$, that is 
\[(\CE^{k,\gamma}_{\CQ|\bx_N}\circ \mathrm{VaR}_\beta)(X) = \infty\]
for every set of observations $\bx_N$ and all choices of $k,\gamma<\infty$.
\end{lemma}
\begin{proof}
 As in the Expected Shortfall case, we observe that $\beta^{-1/\theta} \to \infty$ as $\theta\downarrow 0$, but $\alpha_{\CQ|\bx_N}(\theta)\not\to \infty$  as $\theta\to 0$. The result follows.
\end{proof}

\begin{remark}
 If we assumed Assumption \ref{assn:Pareto}(ii), then for $\beta\to 0$ and finite $N$, we would obtain $(\CE^{k,\gamma}_{\CQ|\bx_N}\circ \mathrm{VaR}_\beta)(X) \approx 1/\beta$, independently of the observed values. Clearly this is not reliable statistically, as it is the assumption of integrability, rather than the observations, which is leading to finiteness of the estimate.
\end{remark}

\begin{theorem}
Under Assumption \ref{assn:Pareto}(i), for all $\gamma<\infty$, the DR-Value at Risk with tail probability $\beta$ is regularizable whenever $\beta^{-1}= O(1)$ (as $N\to\infty$), for some $\eta>0$.
\end{theorem}
\begin{proof}
 As in the Expected Shortfall case, we see that 
\[\frac{d^n}{d\theta^n}\Big(\beta^{-1/\theta}\Big) = O(\beta^{-1/\theta} ((\log \beta)^n+1)).\]
Applying Theorem \ref{thm:approxthm} leads to the proposed condition for regularizability (in the case $\gamma=1$)
\[\frac{1}\beta = o_P(N^{\hat\theta-\eta})\qquad \text{for some }\eta>0, \text{ for all }P\in\CQ.\]
Again, using Lemma \ref{lem:delta0reg}, for regularizability of the $\gamma>1$ case, we obtain the condition $\beta^{-1} = o_P(N^{\hat\theta/2-\eta})$ for some $\eta>0$. Consistency of the MLE (and the assumption $\theta>0$ for all $P$), shows that this is guaranteed when $\beta^{-1}= O(1)$ as $N\to\infty$.
\end{proof}
\begin{remark}
 The requirement $\beta^{-1}=O(1)$ is quite restrictive, but comes from the fact we are assuming nothing beyond Assumption \ref{assn:Pareto}(i), i.e. that our distribution is well defined. Strengthening Assumption \ref{assn:Pareto}(i) to restrict to $\theta>\tilde\theta$ for some $\tilde\theta>0$, we have regularizability whenever $\beta^{-1} = o(N^{\tilde\theta/2-\eta})$ for some $\eta>0$.
\end{remark}

\subsection{Probability of loss}

The probability of a loss exceeding a level $\beta$ is given by (under Assumption \ref{assn:Pareto}(i)
\[\mathrm{PL}_\beta(X)=P(X>\beta) = 1-F_\theta(\beta)=\beta^{-\theta}.\]
This is not a convex expectation, however we can equivalently express it as $P(X>\beta)=E[I_{X>\beta}]$, and then consider the regularity of 
\[(\CE^{k,\gamma}_{\CQ|\bx_N} \circ \mathrm{PL}_\beta)(X)= \CE^{k,\gamma}_{\CQ|\bx_N}(I_{X>\beta}).\] 
\begin{theorem}
 The DR-probability of a loss is regular, for all $N$, $\gamma$ and $\beta$.
\end{theorem}
\begin{proof}
As $I_{X>\beta}$ is bounded, this is well behaved without regularization, no matter what the choice of $\beta$. 
\end{proof}

\begin{remark}
The asymptotic behaviour of the DR-probability of loss (and other bounded random variables) is described by \cite[Section 3.2]{Cohen2016}.
\end{remark}

\subsection{Integrated tail and Cram\'er--Lundberg failure probability}

From an insurance perspective, it is sometimes of interest to look at the integrated tail, which under Assumption \ref{assn:Pareto}(ii) is given by
\[\mathrm{IT}_\beta(X) := E[(X-\beta)^+] = \int_\beta^\infty (1-F_\theta(x))dx = \int_\beta^\infty x^{-\theta}dx = \frac{\beta^{1-\theta}}{\theta-1}.\]
\begin{remark}
For $\beta\geq1,$ this a convex map, but not translation invariant, so the DR-integrated tail, $(\CE_{\CQ|\bx_N}^{k, \gamma}\circ \mathrm{IT}_\beta)(\cdot)$, is not generally a convex expectation.
\end{remark}
As in the previous cases, without truncation, the DR-integrated tail poses some problems.
\begin{lemma}
 Without truncation, under Assumption \ref{assn:Pareto}(ii), the DR-integrated tail is infinite for all $N$, and all $\beta\geq 1$, that is 
\[(\CE^{k,\gamma}_{\CQ|\bx_N}\circ \mathrm{IT}_\beta)(X) = \infty\]
for every set of observations $\bx$ and all choices of $k,\gamma<\infty$.
\end{lemma}
\begin{proof}
   By direct calculation, considering $\theta\downarrow 1$,
\[(\CE^{k,\gamma}_{\CQ|\bx_N} \circ \mathrm{IT}_\beta)(X) = \sup_{\theta} \Big\{\frac{\beta^{1-\theta}}{\theta-1} - \frac{N}{k}\Big(\log(\hat \theta/\theta) -1+ \theta/\hat \theta\Big)\Big\}=\infty.\]
\end{proof}

The surprising result is that, while we need to truncate to avoid infinite values, the DR-integrated tail is always regularizable.
\begin{theorem}
 Under Assumption \ref{assn:Pareto}(ii), for all $\gamma<\infty$, the DR-integrated tail is regularizable for all choices of $\beta\ge 1$ (with any desired dependence on $N$).
\end{theorem}
\begin{proof}
 For all $P\in\CQ$, as in the earlier settings, we calculate
\[\frac{d^n}{d\theta^n}\Big(\frac{\beta^{1-\theta}}{\theta-1}\Big)=O\big(\beta^{1-\theta}((\log\beta)^n+1)\big).\]
However, for $\beta\geq 1$, $\theta>1$, the right hand side is $o(1)$ with respect to $\beta$, in particular (by consistency of the MLE) it is $o_P(1)$ for every $P\in \CQ$. Therefore, the integrated tail is regularizable with no restriction on the value of $\beta$, for every value of $\gamma<\infty$.
\end{proof}

A closely related quantity is the related failure probability under a Cram\'er--Lundberg model, given by (under Assumption \ref{assn:Pareto}(ii))
\[\mathrm{CL}_\beta(X) = \frac{\mathrm{IT}_\beta}{E[X]}  = \frac{\beta^{1-\theta}}{\theta}.\]
\begin{theorem}
 For all choices of $\beta\geq 1$, under Assumption \ref{assn:Pareto}(ii), the DR-Cram\'er--Lundberg failure probability is regularizable.
\end{theorem}
\begin{proof}
 We calculate
\[(\CE^{k,\gamma}_{\CQ|\bx_N} \circ \mathrm{CL}_\beta)(X) = \sup_{\theta} \Big\{\frac{\beta^{1-\theta}}{\theta} - \frac{N}{k}\Big(\log(\hat \theta/\theta) -1+ \theta/\hat \theta\Big)\Big\}.\]
For $\beta\geq 1$, we know $\frac{\beta^{1-\theta}}{\theta}$ is (uniformly) bounded for all $\theta>1$. Therefore, from the consistency of the MLE and Theorem \ref{thm:approxthm}, $(\CE^{k,\gamma}_{\CQ|\bx_N} \circ \mathrm{CL}_\beta)(X)$ is regularizable.
\end{proof}

\subsection{Distortion risk}
As a final example, we consider a distortion based nonlinear expectation. This is given by taking a convex increasing bijective map $\lambda:[0,1]\to[0,1]$, then calculating the expectation with the transformed cdf
$F_\lambda(x) = \lambda(F(x))$. As $\lambda$ is convex, it is differentiable almost everywhere, and we can calculate the transformed density under Assumption \ref{assn:Pareto},
\[f_\lambda(x) = \lambda'(F(x)) f(x) = \lambda'(1-x^{-\theta})\theta x^{-(1+\theta)}.\]
so the distortion risk is given by 
\begin{equation}\label{eq:distortion}\mathrm{D}_\lambda(X) = \int_{1}^\infty \lambda'(1-x^{-\theta})\theta x^{-\theta}dx.\end{equation}
\begin{lemma}
 The distortion risk is finite whenever there exists $\zeta>1/\theta$ such that
\[\lambda'(y) = O((1-y)^{-1+\zeta}) \quad \text{ as }y\to 1.\]
\end{lemma}
\begin{proof}
In order for $\mathrm{D}_\lambda$ to be finite, as $\theta>0$, for large $x$ we require $\lambda'(1-x^{-\theta})$ not to be too large, else we lose integrability in \eqref{eq:distortion}. As we know $x^{-(1+\zeta)}$ is integrable on $[1,\infty[$, the result follows by dominated convergence.
\end{proof}
As $\lambda$ is increasing and convex, for any $y<1$, we know that $\{\lambda'(x)\}_{x<y}$ is bounded, so it is the behaviour of $\lambda'$ near $1$ which is of interest. For this reason, we will focus our attention on the following example.
\begin{definition}
 We say $\lambda$ is the `minmaxvar' transform (with parameter $\beta \in (0,1]$) if $\lambda(x) = \lambda_\beta(x):= 1-(1-x)^\beta$.
\end{definition}
This case is interesting from our perspective, as it describes the critical growth of $\lambda$ near the boundary $x=1$. Other classic examples, for example the Wang transform $\lambda(x) = \Phi(\Phi^{-1}(x)-\beta)$, where $\Phi$ is the normal cdf, are also of interest in some settings, but do not have this critical growth. This is closely related to the `minmaxvar' risk measure considered by Cherny and Madan \cite{Cherny2009}.

\begin{lemma}
Without truncation, under Assumption \ref{assn:Pareto}(ii),  the DR-minmaxvar risk is infinite for all $N$, and all $\beta\in(0,1)$, that is 
\[(\CE^{k,\gamma}_{\CQ|\bx_N}\circ \mathrm{D}_{\lambda_\beta})(X) = \infty\]
for every set of observations $\bx_N$ and all choices of $k,\gamma<\infty$.
\end{lemma}
\begin{proof}
We know $\lambda'_\beta(x) = \beta (1-x)^{\beta-1}$, so
\[\mathrm{D}_{\lambda_\beta} = \int_{1}^\infty \lambda'_\beta(1-x^{-\theta})\theta x^{-\theta}dx = \int_{1}^\infty \beta (x^{-\theta})^{\beta-1}\theta x^{-\theta}dx = \frac{\beta\theta}{\beta\theta-1}\]
provided $\beta>1/\theta$, and is otherwise infinite. As $\alpha_{\CQ|\bx}(1/\beta) <\infty$ for all $\beta \in(0,1)$, our DR-minmaxvar risk will be infinite.
\end{proof}

\begin{theorem}
 Under Assumption \ref{assn:Pareto}(ii), the DR-minmaxvar risk is not regularizable for any $\beta<1$. However, under the stronger assumption that we restrict our models to those where $\theta>\tilde \theta$ for some $\tilde\theta>1$, then the DR-minmaxvar risk is regularizable whenever
\[\beta\geq \frac{1}{\tilde \theta} + \frac{1}{|O(N^{\eta})|}.\]
where $\eta=1/4$ in the case $\gamma=1$, and $\eta<1/4$ in the case $\gamma<\infty$.
\end{theorem}
\begin{proof}
For any fixed $\beta<1$, Assumption \ref{assn:Pareto}(ii), is insufficient to guarantee that the MLE $\hat \theta>1/\beta$ (as $\hat\theta$ may be arbitrarily close to $1$). As this is the condition for finiteness of $\mathrm{D}_{\lambda_\beta}$ and the MLE satisfies $\alpha_{\CQ|\bx_N}(\hat\theta)\equiv0$, the non-regularizability follows.

Under our stronger assumption, to determine conditions on $\beta$ such that the DR-minmaxvar risk is regularizable, we proceed as before. We can calculate
\[(\CE^{k,1}_{\CQ|\bx_N} \circ \mathrm{D}_{\lambda_\beta})(X) = \sup_\theta\Big\{\frac{\beta\theta}{\beta\theta-1} - \Big(\frac{N}{k}\Big(\log(\hat \theta/\theta) -1+ \theta/\hat \theta\Big)\Big)^\gamma \Big\}.\]
 We have the derivatives
\begin{align*}
 \frac{d}{d\theta}\Big(\frac{\beta\theta}{\beta\theta-1}\Big)&=\frac{-\beta}{(\beta\theta-1)^2},\\
 \frac{d^2}{d\theta^2}\Big(\frac{\beta\theta}{\beta\theta-1}\Big)&=\frac{2\beta^2}{(\beta \theta-1)^3},\\
 \frac{d^3}{d\theta^3}\Big(\frac{\beta\theta}{\beta\theta-1}\Big)&=\frac{-6\beta^3}{(\beta\theta-1)^4}.
\end{align*}
With $g_N(\psi) = \frac{\beta(\hat\theta+\psi)}{\beta(\hat\theta+\psi)-1} - \frac{\beta\hat\theta}{\beta\hat\theta-1}$, Theorem \ref{thm:approxthm} is satisfied by assuming $(\hat\theta\beta -1)^{-4}=O_P(N)$, or equivalently, 
\[\beta \geq \frac{1}{\hat\theta} + \frac{1}{|O_P(N^{1/4})|}.\]
Using the consistency of the MLE we know $\hat\theta>\tilde\theta$, and the result follows. If we make the further assumption that $\beta=\tilde\theta^{-1}+ 1/|O(N^{\eta})|$ for $\eta<1/4$, we can take $\epsilon=1/2$ in Theorem \ref{thm:approxthm}, to obtain regularizability for all $\gamma<\infty$.
\end{proof}

\begin{remark}
 Many further cases can also be considered, using this general approach. It would be interesting also to apply DR-regularization in a more general estimation problem, where we do not assume simply that we have a Pareto distribution, but must use a generalized extreme value or generalized Pareto model, with the associated estimation difficulties. 
\end{remark}

\section*{Appendix}

 To prove our key approximation result (Theorem \ref{thm:approxthm}) we begin with the following lemma.
 \begin{lemma}\label{abstractapprox}
  Consider a $C^3$ function $f:\bR^N\to \bR$, with negative definite Hessian at zero $H$. Suppose we wish to find a local maximum in a small ball $B_\delta$ of radius $\delta$ around $0$. Let $c_\delta = \sup_{x\in B_\delta}\|f'''(x)\|$. If a local maximum exists in the interior of the ball, then it has position
 \[x^*= - H^{-1} f'(0) + c_\delta \|H^{-1}\|O(\delta^2).\]
  \end{lemma}
 \begin{proof}
  By Taylor's theorem, we can write the expansion
 \[f(x) = f(0)+ x^\top f'(0) + \frac{1}{2}x^\top H x + R(x)\]
 where $R$ is some remainder term with  $|R(x)|\leq (c_\delta/6)\|x\|^3$ and derivative $\|R'(x)\|\leq (c_\delta/2)\|x\|^2=c_\delta O(\delta^2)$ on our ball. To find a local extremum $x^*$, we differentiate to obtain the vector equation
 \[0= f'(0) + H x^* +  R'(x^*)\]
 which rearranges to give the approximation (which is true for every interior local extremum)
 \[x^* = - H^{-1} f'(0) - H^{-1} R'(x^*).\]
 For an extremum within the ball, we have the desired approximation
 \[x^*= - H^{-1} f'(0) + c_\delta\|H^{-1}\|O(\delta^2).\]
 \end{proof}

We now combine Lemma \ref{abstractapprox} and Lemma \ref{quadpenbound} to give a proof of Theorem \ref{thm:approxthm}, which we repeat here for the ease of the reader.
\begin{theorem}
Consider the maximization of the general function 
\[f_N(\psi) := g_N(\psi) - \frac{1}{k}\alpha_{\CQ|\bx_N}(\hat\theta+\psi)\]
Suppose that $\alpha_{\CQ|\bx_N}$ arises from a setting where the result of Lemma \ref{quadpenbound} can be applied and, for some $\epsilon\in\,]0,1/2]$, for every $P\in\CQ$, for every $N$ the function $g_N$ is $C^3$ and satisfies, for $B$ a ball of constant radius around $\psi=0$,
\begin{enumerate}[(i)]
\item $g_N(0)=0$,
 \item  $\|g_N'(\psi)\|=o_P(N^{1-\epsilon})$ uniformly on $B$,
  \item $\|g_N''(0)\| = o_P(N)$,
  \item $\|g_N'''(\psi)\| = O_P(N)$ uniformly on $B$.
\end{enumerate}
Then there exists a point $\psi^*_N$ such that, for each $P\in\CQ$, with $P$-probability approaching $1$ as $N\to\infty$, we know that $\psi^*_N$ maximizes the value of $f_N$ on the set  $\{\psi:\alpha_{\CQ|\bx_N}(\hat\theta+\psi)\leq N^{1-2\epsilon}\}$ and, furthermore, $\alpha_{\CQ|\bx_N}(\hat\theta+\psi^*_N)= o_P(N^{1-2\epsilon})$.
\end{theorem}

\begin{proof}
By Lemma \ref{quadpenbound}, for all $N$ sufficiently large, with arbitrarily high $P$-probability, for $\psi$ in a neighbourhood of zero of fixed radius, there exists a constant $C>0$ (depending on $P$) such that
\[\alpha_{\CQ|\bx_N}(\hat\theta+\psi) > \frac{N}{C}\|\psi\|^2\]
and outside this neighbourhood $\alpha_{\CQ|\bx_N}(\hat\theta+\psi)\geq O(N)$. In all that follows, we restrict our attention to this constant radius ball.

Omitting subscript $N$ for simplicity, from our assumptions on $g$ we know
\[g(\psi) = o_P(N^{1-\epsilon})\|\psi\|\]
Consequently, except possibly on a ball of radius $O_P(N^{-\epsilon})$ around zero, we know that 
\[\alpha_{\CQ|\bx_N}(\hat\theta+\psi)> \frac{N}{C}\|\psi\|^2> o_P(N^{1-\epsilon})\|\psi\| = |g(\psi)|.\]
It follows that, for all $\psi$ outside a ball of radius $O_P(N^{-\epsilon})$, we know that $f(\psi)<0$. As $f(0)=0$, there must exist a local maximum within the ball of radius $O_P(N^{-\epsilon})$.
 
We know that, 
\begin{align*}
f'(0)&=g'(0) = o_P(N^{1-\epsilon})\\
\|f''(0)^{-1}\|&\leq \|( g''(0) - N/C)^{-1}\| = O_P(N^{-1})\\
f'''(\psi) &= g'''(\psi) + \alpha_{\CQ|\bx_N}'''(\hat\theta+\psi) = O_P(N).
\end{align*}
Applying Lemma \ref{abstractapprox}, we know that any local maximum of $f$ within a ball of radius $o_P(N^{-\epsilon/2})$ will be at a point satisfying
 \begin{align*}
  \psi^* &= -f''(0)^{-1} f'(0) + O_P(N) \|f''(0)^{-1}\| o_P(N^{-\epsilon}) = o_P(N^{-\epsilon}).
 \end{align*}
 Therefore, \emph{all} local maxima within the ball of radius  $o_P(N^{-\epsilon/2})$ will be within the ball of radius $o_P(N^{-\epsilon})$.
 
Taking a Taylor approximation of the $C^3$ function $\alpha_{\CQ|\bx_N}$ we see that, within  the ball of radius $o_P(N^{-\epsilon})$, 
\begin{align*}
\alpha_{\CQ|\bx_N}(\hat\theta+\psi) &= N\psi^\top (\mathfrak{I}_{\hat\theta}+O_P(\|\psi\|))\psi = \|\mathfrak{I}_{\hat\theta}\|o_P(N^{1-2\epsilon})  + o_P(N^{1-3\epsilon})\\ &= o_P(N^{1-2\epsilon}) .
\end{align*}

Conversely, outside the ball of radius $o_P(N^{-\epsilon/2})$, we know that $\alpha_{\CQ|\bx_N}(\hat\theta+\psi) > N^{1-\epsilon}/C$ for some $C>0$. Therefore, we can be certain that  a point $\psi^*=o_P(N^{-\epsilon})$ will be the maximizer within the region $\{\psi:\alpha_{\CQ|\bx_N}(\hat\theta+\psi) < N^{1-2\epsilon}\}$, as desired.
\end{proof}

\section*{Acknowledgements}
Research supported by the Oxford--Man Institute for Quantitative Finance and the Oxford--Nie Financial Big Data Laboratory.

\bibliographystyle{plain}  
\bibliography{EVTbib}
\end{document}